\documentclass[12pt]{amsart}
\usepackage{graphicx}
\usepackage{amssymb}
\usepackage{amscd}
\usepackage{hyperref}
\usepackage{verbatim}
\usepackage{amsmath}
\newtheorem{thm}{Theorem}
\newtheorem{theorem}{Theorem}[section]
\newtheorem{corollary}[theorem]{Corollary}
\newtheorem{lemma}[theorem]{Lemma}
\newtheorem{prop}[theorem]{Proposition}

\theoremstyle{remark}
\newtheorem{rem}[theorem]{\bf Remark}

\theoremstyle{definition}
\newtheorem{definition}[theorem]{Definition}

\theoremstyle{definition}
\newtheorem*{ack}{Acknowledgments}

\theoremstyle{remark}

\newcommand{\dbar}{\bar\partial}
\newcommand{\im}{{\rm im\,}}

\newcommand{\dint}{d_\square} 
\newcommand{\dom}{{\rm dom\,}}

\newcommand{\supp}{\, {\rm supp}\,}
\newcommand{\bid}{\, {\bf 1}}
\newcommand{\dpq}{\mathcal D^{p,q}}
\newcommand{\lpq}{\Lambda^{p,q}}

\newcommand{\cbpq}{C^\infty(\bar M,\lpq)}
\newcommand{\ccpq}{C^\infty_c(\bar M,\lpq)}
\newcommand{\lzpq}{L^2(M,\Lambda^{p,q})}

\newcommand{\CC}{\mathbb{C}}
\newcommand{\bC}{\mathbb{C}}

\newcommand{\Hmm}[1]{\leavevmode{\marginpar{\tiny%
$\hbox to 0mm{\hspace*{-0.5mm}$\leftarrow$\hss}%
\vcenter{\vrule depth 0.1mm height 0.1mm width \the\marginparwidth}%
\hbox to 0mm{\hss$\rightarrow$\hspace*{-0.5mm}}$\\\relax\raggedright #1}}}

\begin{document}
\title[Heat estimates on $G$-manifolds]{Heat kernel estimates for the
$\dbar$-Neumann problem on $G$-manifolds}

\author{Joe J Perez}
\address{Fakult\"at f\"ur Mathematik\\
Universit\"at Wien\\
Vienna, Austria}
\email{joe\_j\_perez@yahoo.com}
\thanks{JJP is supported by FWF grant P19667, {\it Mapping Problems in Several
Complex Variables}}

\author{Peter Stollmann}
\address{Fakult\"at f\"ur Mathematik\\
Technische Universit\"at\\
Chemnitz, Germany}
\email{peter.stollmann@mathematik.tu-chemnitz.de}
\thanks{PS is partly supported by DFG}

\begin{abstract}
We prove heat kernel estimates for the $\dbar$-Neumann Laplacian $\square$
acting in spaces of differential forms over noncompact manifolds with a Lie group symmetry and compact quotient. We also relate our results to those for an associated Laplace-Beltrami operator on functions. \end{abstract}
\date{\today}

\maketitle

\section{Introduction}

We are concerned with bounds on the heat kernel of the $\dbar$-Neumann Laplacian
on manifolds with boundary possessing a Lie group symmetry. Heat kernel bounds are an object of intensive study
and an attempt to describe only the most important works would go well beyond
the scope of the present article. Instead we refer to \cite{Grigoryan}  and
point out the pecularities of the model
we are dealing with before properly introducing the setup. The operator we deal
with is the natural Laplacian coming from the PDE of several complex variables.
It acts on complex-valued differential forms on a manifold with boundary and has non-coercive boundary
conditions. Despite these differences from the usual situation, some techniques from the theory of Dirichlet forms remain applicable to
obtain the bounds which frequently are the goal in studies of heat estimates in very different settings. Due to the complications in our model, it comes as a nice surprise that these usual tools, {\it e.g.}\ the intrinsic
metric, come in so handy. Apart from the results that will soon be mentioned,
this surprise is certainly a message we want to pass along. Since we would like to
communicate our results to people in at least two communities, we will take our
time to explain certain basics that might be obvious to some readers. We ask those to
please bear with us.

Let $M$ be a complex manifold, $n=\dim_{\mathbb C} M$, and assume that $M$ has a
smooth boundary $bM$ such that $\bar M = M\cup bM$. Assume further that $\bar M$
is contained in a slightly larger complex manifold $\widetilde{M}$ of the same
dimension. The space of holomorphic functions on $M$ under various complex-geometric conditions on
$bM\subset\widetilde{M}$ has been investigated from various standpoints, beginning with Hartogs and
Levi \cite{Ha, L1,L2} and, with Stein theory and sheaf-theoretic methods,
culminating in the Oka-Grauert theorem, \cite{Gr}.

An approach to problems in several complex variables using partial
differential equations was also developed by Morrey, Spencer,
Andreotti-Vesentini, Kohn, Nirenberg, H\"ormander, and others (\cite{FK, Siu,
Stra}) bearing fruit in Kohn's solution to the $\dbar$-Neumann problem,
\cite{K1, K2}.  This method heavily involves the analysis of a self-adjoint
Laplace operator $\square$ on differential forms in $\Lambda^{p,q}$, the subject
of this article, which we describe here.

For any integers $p,q$ with $1\leq p,q\leq n$ denote by
$C^\infty(M,\Lambda^{p,q})$ the space of all $C^\infty$ forms of type $(p,q)$ on
$M$. These are the differential forms which can be written in local complex coordinates
$(z_1, z_2,\dots,z_n)$ as
\begin{equation}\label{pqform}\phi=\sum_{|I|=p,|J|=q}\phi_{I,J}\ dz^I\wedge
d\bar z^J \end{equation}
\noindent
where $dz^I=dz^{i_1}\wedge\dots\wedge dz^{i_p}$, $dz^J=d\bar
z^{j_1}\wedge\dots\wedge d\bar z^{j_q}$, $I=(i_1,\dots,i_p)$, $J=(j_1,\dots,
j_q)$, $i_1<\dots<i_p$, $j_1<\dots<j_q$, and the $\phi_{I,J}$ are smooth
functions in local coordinates.  For such a form $\phi$, the value of the
antiholomorphic exterior derivative $\dbar\phi$ is
\[\dbar\phi=\sum_{|I|=p,|J|=q} \sum_{k=1}^n
\frac{\partial\phi_{I,J}}{\partial\bar z^k}\ d\bar z^k\wedge dz^I\wedge d\bar
z^J\]
\noindent
so $\dbar = \dbar|_{p,q}$ defines a linear map
$\dbar:C^\infty(M,\Lambda^{p,q})\to C^\infty(M,\Lambda^{p,q+1})$.

With respect to a smooth measure on $M$ and a smoothly varying Hermitian
structure in the fibers of the tangent bundle, define the spaces
$L^2(M,\Lambda^{p,q})$. Let us extend the above $\dbar$ to the corresponding
maximal operator in $L^2$ (and still call it $\dbar$)
and let $\dbar^\ast$ be its adjoint operator (the forms in the domain of
$\dbar^\ast$ will have to satisfy certain boundary conditions). Then
\begin{eqnarray}\label{qform}
\dom(Q^{p,q})&:=& \dom (\dbar)\cap\dom(\dbar^\ast)\\
Q^{p,q}(\phi,\psi)&:=& \langle\dbar\phi,\dbar\psi\rangle_{L^2(M,
\Lambda^{p,q+1})} +
\langle\dbar^*\phi,\dbar^*\psi\rangle_{L^2(M,\Lambda^{p,q-1})},\end{eqnarray}
\noindent defines a closed form $Q^{p,q}$ on $L^2(M,\Lambda^{p,q})$; we will
frequently omit the superscripts indicating the type of forms and
simply write $Q$ and $\dom(Q)$ instead. By standard theory (see details in
Section \ref{forms} below) there is a unique selfadjoint operator
$\square=\square_{p,q}$ corresponding to $Q=Q^{p,q}$ that we can write as
$$\square=\square_{p,q}=\dbar^\ast\dbar+\dbar\dbar^\ast .$$
This Laplacian $\square$ is elliptic but its natural boundary conditions are not
coercive,
thus, in the interior of $M$, the operator gains two degrees in the Sobolev
scale, as a second-order operator, while in neighborhoods of the boundary it
gains less. The gain at the boundary depends on the geometry of the boundary,
and the best such situation is that in which the boundary is {\it strongly
pseudoconvex}, a condition already seen to be important in \cite{Ha, L1,L2}; see \cite{Siu}. 
In that case, the operator gains one
degree on the Sobolev scale and so global estimates including both interior and
boundary neighborhoods gain only one degree. More precisely, one obtains {\it a
priori} (called Kohn-type) estimates of the form
\[\|u\|_{H^{s+1}(M,\Lambda^{p,q})} \lesssim \|\square u\|_{H^s(M,\Lambda^{p,q})} +
\|u\|_{L^2(M,\Lambda^{p,q})},\quad (u\in \dom\square\cap C^\infty)\]
\noindent
 when the boundary
is strongly pseudoconvex and $q>0$ (\cite{P1}, Lemma 7.11). Such estimates are
usually called {\it subelliptic} as the gain of the operator is less than its
order.

Assuming for the moment that $\bar M$ is compact, under various well-investigated conditions on
$bM$, (\cite{D'A, C}, {\it etc.}) the Laplacian satisfies a {\it pseudolocal
estimate with gain $\epsilon>0$ in} $L^2(M,\Lambda^{p,q})$. That is, if
$U\subset\bar M$ is a neighborhood with compact closure, $\zeta, \zeta'\in
C^{\infty}_{c}(U)$ for which $\zeta'|_{{\rm supp}(\zeta)}=1$, and $\alpha|_U\in
H^{s}(U,\Lambda^{p,q})$, then $\zeta(\square +\bid )^{-1}\alpha\in
H^{s+\epsilon}(\bar M,\Lambda^{p,q})$ and there exists a constant $C_{s,\zeta,
\zeta'}>0$ such that
\begin{equation}\label{prima1}\|\zeta (\square
+\bid)^{-1}\alpha\|_{H^{s+\epsilon}(M,\Lambda^{p,q})}\le C_{s,\zeta, \zeta'}
(\|\zeta'\alpha\|_{H^s(M,\Lambda^{p,q})}+\|\alpha\|_{L^2(M,\Lambda^{p,q})})
\end{equation}
\noindent
uniformly for all $\alpha$ satisfying the assumption. Since $\bar M$ is assumed
compact, Rellich's theorem provides that $(\square +\bid)^{-1}$ is a compact
operator and thus there exists an orthonormal basis of $L^2(M,\Lambda^{p,q})$
consisting  of eigenforms of $\square$, \cite[Prop.\ 3.1.11]{FK}. With the eigenvalues and eigenforms of $\square$, one can construct the heat operator and study it. In our case of noncompact $M$ we will take a different approach. Still, to us, the most important result from the PDE of several complex variables is that a pseudolocal estimate \eqref{prima1} holds even without assuming the
compactness of $M$, as shown in \cite{E}.

\medskip

{\bf Main results:} We will assume throughout this article that we have a complex manifold $M$
which is the total space of a principal bundle on which a Lie group $G$ acts by
holomorphic transformations with compact orbit space $\bar X = \bar M/G$:
\[G \longrightarrow M \longrightarrow X.\]
Throughout this article, we will
assume for simplicity that our manifolds are
strongly pseudoconvex. That implies that a pseudolocal estimate
with gain $\epsilon=1$ holds in $L^2(M,\Lambda^{p,q})$ for all $q>0$. All the
bundles constructed in \cite{HHK} and are treated in \cite{GHS} are strongly pseudoconvex. In our results, one can revert to the more general setting, in which $0<\epsilon<1$, making inessential
changes.

The first of our principal results is a Nash-type inequality, {\it cf.}\ \cite{N}:
\begin{thm}\label{thm1.1}{\rm\bf (Nash inequality)} Let $M$ be a strongly
pseudoconvex $G$-manifold on which $G$ acts freely by holomorphic
transformations with compact quotient $\bar M/G$. For integer $s>\dim_{\mathbb
C}M$
\[\| u\|_{\lzpq}^{2+\frac{1}{s}}\lesssim Q(u) \|
u\|_{L^1(M,\lpq)}^{\frac{1}{s}}, \quad (u\in\dom(Q^{p,q})\cap L^1(M,
\Lambda^{p,q})).\]
\end{thm}

\noindent
Defining the heat semigroup by $P_t = e^{-t\square}$, we obtain operator norm
estimates in $L^p$ spaces as well as Sobolev spaces:
\[\|P_t\|_{L^2\to L^\infty}, \qquad \|P_t\|_{L^1\to L^\infty}, \qquad
\|P_t\|_{H^r\to H^s},\]
valid for $t>0$, $r,s\in\mathbb R$. This last property can be used to obtain that the Schwartz kernel of the heat operator is smooth for $t>0$.

We also obtain an off-diagonal estimate
for the heat semigroup in terms of the intrinsic metric $\dint$ induced by
$\dbar:C^\infty(M,\mathbb R)\to C^\infty(M,\Lambda^{0,1})$ and a $G$-invariant Hermitian
structure on $\Lambda^{0,1}$. It turns out that $\dint$ is equivalent to the intrinsic
metric $d_{LB}$ induced by the Laplace-Beltrami operator of a Riemannian metric simply related 
to the metric on $\Lambda^{0,1}$.

The off-diagonal estimate is 

\begin{thm} \label{thm1.2}{\rm\bf (Off-diagonal heat kernel estimate)} Let $M$
be as above. For measurable subsets $A, B$ of $M$ it follows
that the heat
semigroup satisfies
\[\| \bid_B P_t \bid_A\|_{2\to 2}
\le\exp\left[-\frac{\dint (A;B)^2}{4t}\right].\]
\end{thm}

\noindent
Note the following peculiarity: as already pointed out, the $\dbar$-Neumann problem is not
elliptic in the sense that inverse of $\square$ does not gain two degrees in the Sobolev scale.
This is due to the boundary conditions, which, even in the
strictly pseudoconvex case, give a gain of only one order of differentiability. Our
method of proof does not make use of the better estimates that are valid in the
interior, where the gain is two as in \cite[Thm.\ 2.2.9]{FK}. The resulting Sobolev
estimates make our Nash inequality somewhat weaker than what
would be true for an elliptic operator with coercive boundary conditions. 

On the other hand, the off-diagonal bound is not affected at all by this. The intrinsic
metric gives just the kind of decay that one would expect for an elliptic problem.

Part of what is happening here is that the pseudolocal estimate that we use is given in terms of isotropic Sobolev norms while the problem is inherently anisotropic. In the compact case, anisotropic estimates have 
been worked out \cite{FS, GS} and it happens that the Laplacian gains two orders
of differentiability in all directions except one ``bad'' direction in the boundary. 

The pseudolocal and Kohn-type estimates that we use here were developed in the noncompact case in \cite{E, P1} and applied in \cite{P2} and \cite{DSP} to construct $L^2$ holomorphic functions in some cases, in a manner analogous to that of Kohn and Gromov, Henkin, Shubin, \cite{K1, K2, FK, GHS}. This last reference contains other examples (regular covering spaces of compact, strongly pseudoconvex complex manifolds and two nonunimodular $G$-manifolds) to which our methods here apply.

The spectral theory of the $\dbar$-Neumann problem has been
previously investigated in \cite{M, F1, F2} in the compact
situation and in \cite{St, ST, BGS, BeS1, BeS2}, methods involving pseudodifferential operators
are brought to bear on the problem, still in the compact case. In \cite{DGSC}, heat kernel asymptotics are developed for subelliptic operators on noncompact groups. In \cite{MO}, an asymptotic expansion is developed for the heat kernel of a general elliptic operator with noncoercive boundary conditions.

The contents of the rest of this article are as follows. In Section 2 we will describe the
basic constructions on $M$ and review the principal properties of the
$\dbar$-Neumann problem relevant to our investigation. Also, we will draw the
more directly accessible conclusions of these properties. In Section 3 we
describe the intrinsic geometry carried by $M$ and derive the heat estimates for
the $\dbar$-Neumann Laplacian.

\section{The $\dbar$-Neumann problem}
\subsection{The geometry}\label{geometry} We will introduce some complex-geometric concepts in this section, basically following
\cite{FK}; see also \cite{Kob,Kod}. On a real, $2n$-dimensional $C^\infty$ manifold $M$, an \emph{almost complex structure on} $M$ is a
splitting of the complexification $TM\otimes_{\mathbb R} \bC$ of the real tangent
bundle $TM$,
$$
TM\otimes_{\mathbb R} \bC=T_{1,0}M\oplus T_{0,1}M,
$$
with the following property; denoting the projections onto $T_{1,0}M$ and
$T_{0,1}M$ by $\Pi_{1,0}$ and $\Pi_{0,1}$, respectively:
\begin{equation}\label{bar}\Pi_{0,1}\zeta = \overline{\Pi_{1,0}\bar \zeta},\end{equation}
where $\ \bar{\ }\ $ denotes complex conjugation.

We can also describe an  almost complex structure by a
fibrewise linear mapping $J:TM\to TM$ with $J^2=-\bid$. These two descriptions are
related via:
\begin{equation}\label{T10}T_{1,0}M=\{ X-iJX\mid X\in TM\}=\ker(J-i)\end{equation}
and
\begin{equation}\label{T01}T_{0,1}M=\{ X+iJX\mid X\in TM\}=\ker(J+i),\end{equation}
see \cite[Chapter I, \S7]{Kob}. For a vector field $X\in TM$, a complex vector field in $TM\otimes_{\mathbb R}\mathbb C$ of the form $X- iJX\in T_{1,0}$ is called a holomorphic vector field while one of the form $X+iJX\in T_{0,1}$
is called antiholomorphic.

Dually, the projections $\Pi_{0,1}$, $\Pi_{1,0}$ induce a splitting of the exterior powers of the complexified cotangent bundle, $\Lambda^k T^*M\otimes_{\mathbb R}\mathbb C$ into holomorphic and antiholomorphic parts so that $\Lambda^k =
\bigoplus_{p+q=k}\Lambda^{p,q}$. The exterior derivative in $\Lambda^kT^*M$ can be combined with the splittings of the complexified cotangent bundle of $M$ to obtain holomorphic and antiholomorphic exterior derivatives $\partial$ and $\dbar$, respectively. The relations among these operators are given by
$$
\dbar:C^\infty(\bar M, \Lambda^{p,q})\to C^\infty(\bar M, \Lambda^{p,q+1}),
\qquad \dbar \phi = \Pi_{p,q+1}d\phi$$
and
$$
\partial:C^\infty(\bar M, \Lambda^{p,q})\to C^\infty(\bar M, \Lambda^{p+1,q}),
\qquad \partial \phi = \Pi_{p+1,q}d\phi
$$
for $\phi\in C^\infty(\bar M, \Lambda^{p,q})$.

On a complex manifold, it is true that $d=\partial + \dbar$, see \cite[Prop.\ 1.2.1]{FK} and that $\dbar^2=0$, which gives rise to the $\dbar$-\emph{complex},
$$
0\to C^\infty(\bar M, \Lambda^{p,0})\stackrel{\dbar}{\to} C^\infty(\bar M,
\Lambda^{p,1}) \stackrel{\dbar}{\to} \dots\stackrel{\dbar}{\to} C^\infty(\bar M,
\Lambda^{p,n})\to 0
$$
which is the starting point for various cohomology theories due to Dolbeault, Hodge-Kodaira, and unified by Spencer, {\it cf.}\ \cite{KN}. See also \cite{P1} for some results related to our current setting.

\subsection{Sobolev spaces} We will have to describe smoothness of functions,
forms, and sections of vector bundles using $G$-invariant Sobolev spaces which
we describe here.

If $E$ is a vector $G$-bundle over $M$, then we may introduce a $G$-invariant
pointwise inner product structure $\langle \cdot, \cdot \rangle_E$ on $E$.
Together with a $G$-invariant measure on $M$, we define the Hilbert space of
sections of $E$ which we denote $L^2(M,E)$. Note that, in particular, spaces of
sections in natural tensor bundles on a $G$-manifold have natural, invariant
Hermitian structures associated to a Riemannian structure on the underlying
manifold; below we will provide more detail. We denote by $C^\infty(M,\lpq)$ the
space of smooth $(p,q)$-forms on $M$, by $\cbpq$ the subspace of those forms
that can be smoothly extended to $\bar M$ and by $\ccpq$ the subspace of the
latter consisting of those smooth forms with compact support. Given any
$G$-invariant, pointwise Hermitian structure
\[C^\infty(\bar M,\Lambda^{p,q})\ni u,v \longmapsto \langle u(x) ,v(x)
\rangle_{\Lambda_x^{p,q}} \in\CC, \quad (x\in \bar M),\]
\noindent
we define the $L^p$-spaces $L^p(M,\Lambda^{q,r})$ of forms as the completions of
$C_c^\infty(\bar M,\Lambda^{q,r})$ in the norms
\[\|u\|_{L^p(M,\Lambda^{q,r})} = \left[\int_M \langle
u,u\rangle_{\Lambda^{q,r}}^{p/2}\right]^{1/p}.\]
\noindent
As in \cite{Gro, Shu} we may construct appropriate partitions of unity and, by
differentiating componentwise with respect to local geodesic coordinates,
assemble $G$-invariant integer Sobolev spaces $H^s(M,\Lambda^{p,q})$, for
$s=0,1,2,\dots$. Because $X=\bar M/G$ is compact, the spaces
$H^s(M,\Lambda^{p,q})$ do not depend on the choices of an invariant metric
on $M$ or of an invariant inner product on
$\Lambda^{p,q}$. The usual duality relations for $L^p$ spaces hold (polarizing
the above norm) as well as the Sobolev lemma, {\it
etc.} Background on this is provided in \cite{GKS}. There, the Hermitian structure is defined in terms of the Hodge operator so our $\langle u,v\rangle_\Lambda$ translates to $u\wedge\star\bar v$. See p82 of \cite{FK} and Lemma \ref{twist} below.

\subsection{Operators and forms}\label{forms}
As we said in the introduction, $\square$ will be defined in terms of an
associated quadratic form. Good references for background on the general concept of closed
forms and their associated operators are  \cite{F,K,RS}, among others.
Here we will give more details concerning the case at hand and also describe
certain subsets of smooth forms that belong to the respective form and operator
domains.
We begin by collecting some information concerning the building blocks of
$\square$,  $\dbar$ and $\dbar^*$.

\begin{rem} Let $M$ be as above.
 \begin{enumerate}
  \item The maximal operator $\dbar$ in $\lzpq$ is given by:
$\alpha\in \dom(\dbar)$ whenever $\dbar\alpha \in L^2(M,\Lambda^{p,q+1})$ in
the distributional sense. It acts from $\lzpq$ to $L^2(M,\Lambda^{p,q+1})$ and
is a closed operator.
 \item The operator $\dbar^*$ in $\lzpq$ is the adjoint of $\dbar$ (in
$L^2(M,\Lambda^{p,q-1})$); it is given by:
$\alpha\in \dom(\dbar^*)$ whenever there exists $\beta\in
L^2(M,\Lambda^{p,q-1})$ so that
$$
\langle \dbar\gamma, \alpha\rangle_{\lzpq} =\langle \gamma,
\beta\rangle_{L^2(M,\Lambda^{p,q-1})} $$
for all $\gamma\in L^2(M,\Lambda^{p,q-1})$ and $\dbar^*\alpha=\beta$.
 \item Since $\dbar$ is closed, the form
$$\dom (\dbar)\times \dom
(\dbar)\ni (\alpha,\beta)\mapsto \langle
\dbar\alpha,\dbar\beta\rangle_{L^2(M,\Lambda^{p,q+1})}$$
is a closed form in $\lzpq$; cf \cite{F,K}.
 \item Since $\dbar^*$ is closed, the form
$$\dom (\dbar^*)\times \dom
(\dbar^*)\ni (\alpha,\beta)\mapsto \langle
\dbar^*\alpha,\dbar^*\beta\rangle_{L^2(M,\Lambda^{p,q-1})}$$
is a closed form in $\lzpq$, provided, $q\ge 1$.
  \item $Q=Q^{p,q}$ is the sum of the closed forms defined in (3), (4)    above
and therefore a closed form as well for $q\ge 1$. $Q^{p,0}$ is the form defined
in (3).
 \end{enumerate}
\end{rem}
Recall that a closed operator is one whose graph is closed, while a form $Q$ is
\emph{closed}  whenever its domain $\dom(Q)$ is a Hilbert space with respect to
the \emph{form inner product}
$$
(\cdot\mid\cdot)_Q := Q(\cdot,\cdot) + \langle\cdot,\cdot\rangle .
$$
The stage is now set for the first form representation theorem, {\it cf.}\ \cite{K},
that asserts that for every semibounded closed form there is a unique
selfadjoint operator associated with the form. In our case, there is a unique
selfadjoint operator $\square_{p,q}$ associated with $Q^{p,q}$, meaning that
\[\dom(\square_{p,q})\subset \dom(Q^{p,q}) \mbox{  and  }
Q(\alpha,\beta)=\langle \square \alpha,\beta\rangle,\]
\noindent whenever $\alpha\in \dom(\square_{p,q})$ and $\beta\in \dom(Q^{p,q})$.
In fact, more is known:
$$
\dom(\square_{p,q})=\{ \alpha\mid \exists \gamma\in \lzpq\,
\forall\beta\in\dom(Q^{p,q}):\,
Q^{p,q}(\alpha,\beta)=\langle\gamma,\beta\rangle\}
$$
and, obviously, $\gamma=\square_{p,q}\alpha$ is uniquely determined. Moreover,
defining the square root $\square_{p,q}^\frac12$ by the functional calculus, we
have that
\[\dom(Q^{p,q})=\dom(\square_{p,q}^\frac12) \mbox{  and  }
Q(\alpha,\beta)=\langle \square^\frac12\alpha,\square^\frac12\beta\rangle.\]
We note that $\square$ can be seen as the \emph{form sum} of the operators
$\dbar^*\dbar$ and $\dbar\dbar^*$. In fact, the former operator is the
selfadjoint operator associated with the form in part (3) of the preceding
remark and the latter is the selfadjoint operator associated with the form in
part (4) of the preceding remark. In that sense, the formula
$$
\square= \dbar^*\dbar+ \dbar\dbar^*
$$
has now a precise meaning, interpreting the +-sign as the form sum, {\it cf.}\
\cite{F,K}.

In principle, all domain questions are settled now and we have defined the forms
and operators we will be dealing with. However, the results above give a rather
implicit description so it is quite useful to have explicit subspaces of the
operator and form domains given above.

 We speak of a \emph{core} of a form meaning a  subspace of its domain that
is dense in the domain with respect to the form norm. Similarly, a core of an
operator is a subspace of its domain that is dense with respect to the graph
norm.

 The following lemma is  from \cite[Lemma 1.1]{GHS} and \cite[Lemma
2.3.2]{FK}; serves the purpose to get our hands on the smooth elements of
certain form and operator domains.

\begin{lemma}\label{dbar} Let $M$ be as above, let $\vartheta$ be
the formal adjoint operator to $\dbar$, and denote by
$\sigma=\sigma(\vartheta,\cdot)$ its principal symbol.
\medskip
\item{(i)}
$\{u\in C_c^\infty(\bar M, \Lambda^\bullet
)\mid\sigma(\vartheta,d\rho)u|_{bM}=0\}$ is a core for $\dbar^*$ and on this
space $\dbar^*$ agrees with $\vartheta$.
\item{(ii)} $\dpq:= \{u\in C^\infty(\bar M, \Lambda^\bullet
)\mid\sigma(\vartheta,d\rho)u|_{bM}=0\}$ is a core for $Q^{p,q}$.
\item{(iii)} The domains of $\dbar^*$ and $Q$ are preserved by multiplication by
cutoff functions.
\end{lemma}

\subsection{Estimates for the Laplacian}\label{ests} In this section we give our requirements on the boundary geometry and state the pseudolocal estimate in more precise language than in the introduction. As before, assume $M$ to be a complex manifold with nonempty smooth boundary $bM$, $\bar M=M\cup bM$, so that $M$ is the interior of $\bar M$, and ${\rm dim}_{\mathbb C}(M)=n$.  We will also assume for simplicity that $\bar M$ is a closed subset in $\widetilde{M}$, a complex neighborhood of $\bar M$ so that the complex structure on $\widetilde{M}$ extends that of $M$, and every point of $\bar M$ is an interior point of $\widetilde{M}$. Let us choose a smooth function $\rho :\widetilde{M}\to \mathbb R$ so that
\[ M=\{z\mid \rho(z)<0\}, \ \ bM = \{z\mid \rho(z)=0\},\]
\noindent
and for all $x\in bM$, we have $d\rho(x)\neq 0$.
In local coordinates near any $x\in bM$ define the {\it holomorphic tangent plane} to the boundary at $x$ by
\[ T_{x}^{\mathbb C}(bM) = \{w\in \mathbb C^{n}\mid \sum_{k=1}^{n}\left.\frac{\partial\rho}{\partial z^k}\right|_x w^k=0\}\]
and define the Levi form $L_{x}$ by
\[ L_{x}(w,\bar w) =
  \sum_{j, k=1}^{n} \left.\frac{\partial^{2}\rho}{\partial z^j\partial \bar z^k}\right|_x  w^j\bar w^k,\quad (w \in T^{\mathbb C}_x(bM)).\]
Then $M$ is said to be {\it strongly pseudoconvex} if for every $x\in bM$, the form $L_{x}$ is positive definite.

The following theorem will be our principal tool from the PDE of several complex variables.
\begin{theorem}\label{globalize}{\rm\bf (Pseudolocal estimate)} Let $M$ be strongly pseudoconvex, $U$ an open
subset of $\bar M$ with compact closure, and $\zeta, \zeta'\in
C^{\infty}_{c}(U)$ for which $\zeta'|_{\supp(\zeta)}=1$.  If $q>0$ and
$\alpha|_{U}\in H^{s}(U,\Lambda^{p,q})$, then $\zeta(\square +1)^{-1}\alpha\in
H^{s+1}(\bar M,\Lambda^{p,q})$ and there exist constants $C_s>0$ so that
 \begin{equation}\label{prima}\|\zeta (\square
+\bid)^{-1}\alpha\|_{H^{s+1}(M,\Lambda^{p,q})}\le
C_s(\|\zeta'\alpha\|_{H^s(M,\Lambda^{p,q})}+\|\alpha\|_{L^2(M,\Lambda^{p,q})}
).\end{equation}
 \end{theorem}
 \begin{proof}This is Prop.\ 3.1.1 from \cite{FK} extended to the noncompact
case in \cite{E}. \end{proof}
\begin{rem}  Boundary geometries giving more general subelliptic estimates are
harder to define, so we refer the interested reader to \cite{D'A, C} instead of pursuing this issue here. For completeness, we mention that the theorem holds
when $M$ satisfies these weaker estimates, {\it mutatis
mutandis} \cite{E}.

A word on notation: For two norms $\|\cdot\|$ and $|\cdot|$, we write $|\phi|\lesssim\|\phi\|$ to mean that there exists a constant $C>0$ such that $|\phi|\le C\|\phi\|$ for $\phi$ in whatever set relevant to the context. \end{rem}
\begin{corollary} For $s\in\mathbb N$, $q>0$, and $\zeta\in C^\infty_c(\bar M)$,
\begin{equation}\label{fun}\|\zeta (\square + \bid)^{-s}
\alpha\|_{H^s(M,\Lambda^{p,q})} \lesssim
\|\alpha\|_{L^2(M,\Lambda^{p,q})},\qquad (\alpha\in L^2(M,
\Lambda^{p,q})).\end{equation}
\end{corollary}
\begin{proof} By induction. Putting $s=0$ in the theorem, we have
\[\|\zeta (\square +\bid)^{-1}\alpha\|_{H^1}\lesssim
\|\zeta'\alpha\|_{L^2}+\|\alpha\|_{L^2} \lesssim \|\alpha\|_{L^2}, \qquad
(\alpha\in L^2(M)).\]
\noindent
Assuming the result for $s-1$, it follows that $(\square + \bid)^{1-s}\alpha \in
H_{\rm loc}^{s-1}(M)$ for all $\alpha\in L^2(M)$. Applying the theorem to this
form, we have
\[\|\zeta (\square +\bid)^{-1}(\square + \bid)^{1-s}\alpha\|_{H^s}\lesssim
\|\zeta'(\square + \bid)^{1-s}\alpha\|_{H^{s-1}}+\|(\square +
\bid)^{1-s}\alpha\|_{L^2},\]
\noindent
and
\[\|\zeta(\square + \bid)^{-s}\alpha\|_{H^s}\lesssim\|(\square +
\bid)^{1-s}\alpha\|_{L^2} \lesssim \|\alpha\|_{L^2}.\]
\end{proof}
\begin{corollary} \label{corso} Let $M$ be a strongly pseudoconvex $G$-manifold on which $G$ acts freely by holomorphic transformations with compact quotient $\bar M/G$. For integer $s>\dim_{\mathbb C}M$ and $q>0$ we
have the estimate
\begin{equation}\label{coso}\|(\square + \bid)^{-s}
\alpha\|_{L^\infty(M,\Lambda^{p,q})} \lesssim
\|\alpha\|_{L^2(M,\Lambda^{p,q})},\qquad (\alpha\in L^2(M,
\Lambda^{p,q})).\end{equation}
\end{corollary}
\begin{proof} Choose $B\subset\bar M$ compact and sufficiently large so that
$B\cdot G$ covers $\bar M$. This is possible since $\bar X$ is compact. Choose
$\zeta\in C^\infty_c(\bar M)$ such that $\supp \zeta\supset B$ in \eqref{fun}.
Now, the Sobolev lemma provides that if $s>k+m/2$, then $H^s(\mathbb R^m)\subset
C^k(\mathbb R^m)$ and there is a constant $C=C_{s,k}$ such that
\begin{equation}\label{sobolem}\sup_{|\alpha|\le k} \sup_{x\in\mathbb R^m}
|\partial^\alpha u(x)|\le C \|u\|_{H^s(\mathbb R^m)},\end{equation}
\noindent
thus, if we take $s>k + 1/2 \dim_{\mathbb R}M = k + \dim_{\mathbb C}M$, we have
\[\| (\square + \bid)^{-s} \alpha\|_{C^k(\bar M)}\lesssim \|\zeta (\square +
\bid)^{-s} \alpha\|_{H^s} \lesssim \|\alpha\|_{L^2},\qquad (\alpha\in L^2(M))\]
\noindent
by the $G$-invariance of $M$ and our choice of local geodesic
coordinates.\end{proof}

\begin{rem}Note that the
exact invariances furnished by the group action assumed here are not
essential and can be relaxed to assumptions on the uniformity of the estimates in \eqref{prima1}, {\it etc}.\end{rem}

\section{Heat kernel estimates and intrinsic geometry}

\begin{definition} Let $\square=\int_0^\infty \lambda dE_\lambda$ be the
spectral resolution of the Laplacian and for $t>0$ put
\[P_t = \int_0^\infty e^{-t\lambda}dE_\lambda.\]
\noindent
That is, $P_t = e^{-t\square}$, and we would write $P_t^{p,q} =
e^{-t\square_{p,q}}$ to be completely explicit.
\end{definition}
\begin{rem} The semigroup $(e^{-tH};t\ge 0)$ of a selfadjoint operator $H$
contains a wealth of information about its generator $H$ and satisfies the
semigroup property $e^{-(t+s)H}=e^{-tH}e^{-sH}$, see
\cite{Dav, Goldstein-85} for the general theory and \cite{Si} for the case of
Schr\"odinger operators. In the case at hand, where $H\ge 0$, the semigroup
consists of contractions, {\it i.e.}, $\| e^{-tH}\|_{2\to 2}\le 1$. The symbol $\|\cdot\|_{2\to 2}$ denotes the operator norm of an operator from $L^2$ to $L^2$. Similar to
what is known for the Laplacian, the semigroup of the $\dbar$-Neumann Laplacian
$\square$ is {\it ultracontractive}. That is, it maps $L^2$ into $L^\infty$ continuously. This is
equivalent to the validity of a Nash-type inequality and will be discussed below.\end{rem}

\subsection{Ultracontractivity and Nash inequalities}\label{ultranash}
The heat operator's ultracontractivity ({\it i.e.}\ boundedness from $L^2\to L^\infty$) follows immediately from the Sobolev estimate in Cor.\ \ref{corso} above. The proof is formally very similar to that from Davies
\cite{D}. The difference between the two cases is that our basic spaces consist of vector-valued functions and so certain concepts and manipulations are not available. For example, we cannot identify
nonnegative elements or take the absolute value in a naive way.

\begin{prop}\label{ultra} Let $M$ be a strongly pseudoconvex $G$-manifold on which $G$ acts freely by holomorphic transformations with compact quotient $\bar M/G$. For integer $s>\dim_{\mathbb C}M$ and $q>0$, we have
\begin{equation}\label{equltra}\| P_t \alpha\|_{L^\infty(M,\lpq)} \lesssim
\max(1,t^{-s})\|\alpha\|_{L^2(M,\lpq)} ,\quad (\alpha\in
L^2(M,\lpq)).\end{equation}
\end{prop}
\begin{proof} We plug $(\square + \bid)^{s}P_t \alpha$ into the inequality
\eqref{coso} from Cor.\ \ref{corso} and obtain:
\begin{eqnarray*}
 \| P_t \alpha\|_{L^\infty} &=& \| (\square + \bid)^{-s} (\square + \bid)^{s}P_t
\alpha\|_{L^\infty} \\
&\lesssim&  \| (\square + \bid)^{s}P_t \alpha\|_{L^2}\\
&\lesssim& t^{-s} \|\alpha\|_{L^2}
\end{eqnarray*}
\noindent
for any $0<t\le 1$, by functional calculus, since the maximum of the function
$\lambda\mapsto(\lambda+1)^s e^{-\lambda t}$ goes like $t^{-s}$ for $t>0$. This
gives the result for arbitrary $t\ge 0$, as the semigroup is a contraction on
$L^2$.
\end{proof}

We mention here that the usual duality properties of the $L^p$ spaces hold in
our setting, \cite{GKS}.

\begin{corollary} \label{diag} Let $M$ be as in the previous proposition. Then, for
integer $s>\dim_{\mathbb C}M$ and $q>0$ we have
\begin{equation}\label{diagest}\| P_t \alpha\|_{L^\infty(M,\lpq)} \lesssim
\max(1,t^{-2s})\|\alpha\|_{L^1(M,\lpq)},\end{equation}
\noindent
uniformly for $\alpha\in L^1\cap L^2(M,\lpq)$.\end{corollary}
\begin{proof} Since $P_t$ is symmetric, $\| P_t\|_{2\to\infty}=\| P_t\|_{1\to
2}$ by duality, and
\[\| P_t\|_{2\to\infty}\lesssim \max(1,t^{-s}),\]
\noindent
from the previous statement, we have
\[\| P_t\|_{1\to \infty}\le \| P_t\|_{2\to \infty} \| P_t\|_{1\to 2} \le \|
P_{t/2}\|_{2\to\infty}^2 \lesssim\frac{1}{t^{2s}}\]
\noindent
by the semigroup property.\end{proof}

\begin{rem} The basic tool in the estimates to come is the fundamental
theorem of calculus applied to the function $t\mapsto \| P_t u\|_{L^2}^2$ or
variants thereof. This rests on the following immediate consequence of
functional calculus: For any $u\in\dom\square$:
$$
P_t u \in \dom\square \mbox{  and  }\frac{d}{dt}\left[P_t u\right] = -\square
P_t u .
$$
\end{rem}

\begin{prop}\label{derive} Let $M$ be as in the previous proposition. For any
real-valued function $w\in C^\infty(\bar
M)\cap L^\infty(M)$ for which $\langle\dbar w,\dbar w\rangle_{\Lambda^{0,1}}$ is
bounded in $M$ and $u\in\lzpq$,
\[\frac{d}{dt}\| e^w P_t u\|_{L^2(M,\Lambda^{p,q})}^2 = -2\mathfrak{Re}\, Q(P_t
u, e^{2w} P_t u).\]
In particular, for $w=0$ we get:
\[\frac{d}{dt}\|  P_t u\|_{L^2(M,\Lambda^{p,q})}^2 = -2 Q(P_t u).\]
\end{prop}
\begin{proof} For any $t>0$ we have
 \begin{eqnarray*}
 \frac{d}{dt}\| e^w P_t u\|_{L^2}^2 &=&  \lim_{h\to 0}\frac{1}{h}\left[ \langle
P_{t+h}u,e^{2w}  P_{t+h}u\rangle
-\langle P_{t}u,e^{2w}  P_{t}u\rangle\right]\\
&=& \lim_{h\to 0}\left[ \langle \frac{1}{h}( P_{t+h}u-P_t u) ,e^{2w}
P_{t+h}u\rangle
+ \langle e^{2w}  P_{t}u, \frac{1}{h}( P_{t+h}u-P_t u)\rangle\right]\\
&=& \langle -\square P_{t}u,e^{2w}  P_{t}u\rangle+ \langle e^{2w}
P_{t}u,-\square P_{t}u\rangle\\
&=& -Q(P_{t}u,e^{2w}  P_{t}u)-Q(e^{2w}P_{t}u,  P_{t}u) ,
 \end{eqnarray*}
where, in the last step we used that $e^{2w}u$ is in the domain of $Q$, by part
$(iii)$ of Lemma \ref{dbar}.
\end{proof}
%
\begin{proof}[\textbf{Proof of Theorem \ref{thm1.1}}]From Prop.\
\ref{ultra} and duality we get
\[t^{-2s}\| u\|_{L^1}^2\ge \langle P_t u,P_t u\rangle_{L^2}=\| P_t u\|_{L^2}^2.\]
\noindent
We use the fundamental theorem of calculus and the above Prop.\ \ref{derive} in
\begin{eqnarray}
 \ldots &=& \| u\|_{L^2}^2- 2\int_0^tQ(P_s u)ds\notag \\
\label{star} &\ge& \| u\|_{L^2}^2- 2tQ(u)
\end{eqnarray}
\noindent
where, in the last inequality we use the following straightforward consequence
of functional calculus:
\[Q(P_s u)=\| \square^\frac12 e^{-s\square}u\|_{L^2}^2\le \| \square^\frac12
u\|_{L^2}^2.\]
Putting $t=Q(u)^{-\frac{1}{2s+1}}\| u\|_{L^1(M,\lpq)}^{\frac{2}{2s+1}}$ in
\eqref{star} gives the assertion.
\end{proof}

\subsection{The intrinsic metric}\label{intrin}
We will measure the bounds on off-diagonal terms in the heat kernel with respect
to the metric given by
\begin{definition}\label{metric} We define the $G$-invariant pseudo-metric $d_\square$ on
$M$ by
\[ \dint (x,y) = \sup \{ w(y)- w(x) \mid w\in L^\infty\cap C^\infty(\bar
M,\mathbb R),\langle\dbar w,\dbar w\rangle_{\Lambda^{0,1}}\le 1\}.\]
\noindent
The distance between sets is given by
\[\dint (A;B):= \sup \{\inf_B w -\sup_Aw \mid w\in L^\infty\cap C^\infty(\bar
M,\mathbb R),\langle\dbar w,\dbar w\rangle_{\Lambda^{0,1}}\le 1\}\]
\noindent
for arbitrary $A,B\subset {\bar M}$. \end{definition}
The definition above is geared to the intrinsic metric of Dirichlet forms, as
used in slightly different versions, {\it e.g.}\ in \cite{BiroliM-95, DER, ERSZ,
Sturm-94b,Sturm-95b, Sto} as well as the metrics considered in
\cite{FeffermanSC-86, JerisonSC-86, NagelSW-85} and see
\cite{Gromov-81,Gromov-07} as well. Note however, that our
application of this concept is somewhat nonstandard. We use
this metric, defined on functions, to estimate the heat kernels acting on forms!
We now show that the metric above is equivalent to an associated Riemannian distance. To
this end, let us describe the metric structure of $M$ in more detail, in the notation
of Sect.\ \ref{geometry} above.

On the tangent bundle $TM$ of the $2n$-dimensional real $G$-manifold underlying $M$, we have a
$G$-invariant almost-complex structure $J: TM\to TM$, induced by the complex structure on $M$. Assume that
we also have a $G$-invariant Riemannian metric $g$ on $TM$ so that $J$ is an isometry with respect to $g$; $g(X,Y) = g(JX,
JY)$. Note that with respect to any such metric, $X\perp JX$. Indeed,
\[g(X,JX) = g(JX,-X) = - g(JX,X) = - g(X,JX) = 0.\]
\noindent
We may extend any
Riemannian structure for which $J$ is an isometry by complex sesquilinearity
(linear in the first slot, conjugate-linear in the second slot) to obtain
Hermitian inner products which we say are {\it associated to} $g$ in $T_{1,0}, T_{0,1}\subset TM\otimes_{\mathbb R}\mathbb C$:
\[\langle X- iJX, Y - iJY\rangle_{T_{1,0}} :=  g(X,Y) + i g(X,JY),\]
\[\langle X+ iJX, Y+ iJY\rangle_{T_{0,1}} :=  g(X,Y) + i g(JX,Y).\]
\noindent
By duality, these structures extend naturally to $\Lambda^{1,0}$ and $\Lambda^{0,1}$ and by tensoriality to
each of the spaces $\Lambda^{p,q}$. We will also metrize the bundle of complex $k$-forms as an
orthogonal sum
\begin{equation}\label{kforms}\Lambda^k =
\bigoplus_{p+q=k}\Lambda^{p,q}, \qquad (k=0,1,\dots, n).\end{equation}
\noindent
Let us describe the $(0,1)$-forms in terms of $J$ analogously to our vector
fields in \eqref{T10}, \eqref{T01}. Since $\Lambda^{0,1}$ is the dual of $T_{0,1}$ in the Hermitian metric
above, we have $\xi_X\in\Lambda^{0,1}$, the dual of $X+iJX\in T_{0,1}$,
naturally of the form
\begin{align}\label{dual}\xi_X(Y+iJY) =& \langle Y+ iJY, X+
iJX\rangle_{T_{0,1}}\\
 =&  g(Y,X) + i g(JY,X)  = g(X,Y) - i g(JX,Y) .\notag\end{align}
We compute the last term in coordinates. Since by assumption we have $g(X,Y) =
g(JX,JY)$, it is true that
\[ g_{kl} J^k_i J^l_j = g_{ij},\]
\noindent
with the convention that repeated indices be summed over. Multiplying this
identity by $J$ and using $J^l_j  J^j_k = - \delta^l_k$, the Kronecker $\delta$,
we get
\[g_{kj}J^j_i = -g_{ij} J^j_k,\]
\noindent
from which it follows that $g(JX,\cdot) = - Jg(X,\cdot)$ since
\[g(JX,\cdot) = g_{ij} J^j_k X^k dx^i \quad {\rm and} \quad Jg(X,\cdot) = J^j_i
g_{jk} X^k dx^i.\]
\noindent
Going back to \eqref{dual}  and writing $Jg(X,\cdot) |_{Y}$ too simply
``$Jg(X,Y)$,'' we see that
\[\xi_X(Y+iJY) =  g(X,Y) + i Jg(X,Y)\]
\noindent
thus $\Lambda^{0,1}\ni\xi_X = \phi_X + iJ\phi_X$ for the real 1-form
$\phi_X = g(X,\cdot)$. Similarly, a form $\Lambda^{1,0}\ni\xi_X = \phi_X -
iJ\phi_X$ again for the real 1-form $\phi_X = g(X,\cdot)$

Now we return to the description of the intrinsic metric. For $w\in
C^\infty(\bar M, \mathbb R)$, consider the following computation:
\[\langle dw, dw\rangle_{\Lambda^1} = \langle(\bar\partial + \partial)w,
(\bar\partial + \partial)w\rangle_{\Lambda^1}  = \langle\bar\partial w,
\bar\partial w\rangle_{\Lambda^{0,1}} + \langle\partial w, \partial
w\rangle_{\Lambda^{1,0}}\]
\noindent
since $\bar\partial w\in\Lambda^{0,1}$ and $\partial w\in\Lambda^{1,0}$ are
orthogonal by the decomposition \eqref{kforms}.

Now, $w$ is real so $\bar\partial w$ is the complex conjugate of $\partial w$ by \eqref{bar},
thus there is a {\it single} real 1-form $\phi$ such that $\bar\partial w =
\phi + iJ\phi$ and $\partial w = \phi - iJ\phi$. In fact, $\phi=\frac12 dw$ since $d=\partial + \dbar$.
Computing the inner products,
\[\langle\bar\partial w, \bar\partial w\rangle_{\Lambda^{0,1}}  =
\langle\partial w, \partial w\rangle_{\Lambda^{1,0}} =2 g(\phi,\phi)\]
\noindent
since $g(\phi, J\phi) = 0$. Thus $\langle dw, dw\rangle_{\Lambda^1} =
2\langle\bar\partial w, \bar\partial w\rangle_{\Lambda^{0,1}} =
4g(\phi,\phi)$ in our metric.

Since the Laplace-Beltrami operator on functions is induced by the quadratic form $w\mapsto \int\langle dw, dw\rangle_{\Lambda^1}$, {\it cf.}\ \cite{Saloff-Coste-92,Sturm-95b}, we have shown

\begin{prop}\label{equiv} For a $J$-invariant Riemannian structure $g$, let $\Delta_{LB}$ be
the corresponding Laplace-Beltrami operator. Given the Hermitian structure on
$\Lambda^{0,1}$ associated to $g$, the intrinsic metric $\dint$ is
equivalent to the one induced by the intrinsic metric of $-\Delta_{LB}$ on functions.
\end{prop}

\begin{rem} (1) At least in the case of complete manifolds without boundary it
is well-known, {\it cf.}\ \cite{Sturm-95b} that the intrinsic metric $d_{LB}$ of the
Laplace-Beltrami operator coincides with the Riemannian distance, i.e.,
$$
 d_{LB}(x,y)=\inf \{ L(\gamma)\mid \gamma:I\to M \mbox{ a curve joining }x,y\in
M\} .
$$
In view of \cite{Sto}, the presence of a boundary should not change this
picture.

(2) For K\"ahler manifolds, $\square=\frac12\Delta$, {\it cf.}\ \cite[Chap.\ III, \S 2]{Kob}, acting componentwise on forms, therefore it is clear in this case that we
recover the intrinsic metric of the Laplacian up to a factor of $\sqrt{2}$.
\end{rem}

\subsection{Off-diagonal heat kernel estimates} Here, we basically use the proof
from \cite{ERSZ}, pointing out once more that
our setup is substantially different as our spaces are spaces of differential forms rather
than functions. Let us also remind the reader that multiplication by functions preserves the
domain of $Q$ and this is crucial to our treatment.
\begin{lemma}\label{twist} For $w\in L^\infty\cap C^1(\bar M,\mathbb R)$, we have
\begin{align} Q(u,u)=  Q(e^{-\epsilon w}u,e^{\epsilon w}u)  - & 2i\epsilon\,
\mathfrak{Im}\, \left\{\langle\dbar u,\dbar w\wedge u\rangle_{L^2} +
\langle\star(\partial w \wedge\star u),\dbar^*u\rangle_{L^2}\right\} \notag \\
& +\epsilon^2 \left\{ \|\dbar w \wedge u\|_{L^2}^2 + \|\partial w \wedge\star
u\|_{L^2}^2\right\} \notag\end{align}
\noindent
for all $u\in\dom Q$. \end{lemma}
\begin{proof} By definition,
\begin{align}Q(e^{-\epsilon w}u,e^{\epsilon w}u) &= \langle\dbar e^{-\epsilon w}
u,\dbar e^{\epsilon w}u\rangle + \langle\dbar^*e^{-\epsilon w}u,\dbar^*
e^{\epsilon w}u\rangle.\notag\end{align}
\noindent
The first term simplifies as follows
\[ \langle\dbar e^{-\epsilon w} u,\dbar e^{\epsilon w}u\rangle =  \langle \dbar
u, \dbar u\rangle + 2 i\epsilon\, \mathfrak{Im}\langle\dbar u, \dbar w\wedge
u\rangle -\epsilon^2 \langle\dbar w\wedge u, \dbar w\wedge u\rangle.\]
\noindent
For the second term, note that $\dbar^* = - \star\partial\star$ where $\star$ is
the Hodge operator and $\partial = d-\dbar$, ({\it cf}.\ Prop.\ 5.1.1,
\cite{FK}). Thus
\[\dbar^* e^{-w}u = - \star\partial\star(e^{-w}u) = -\star [\partial
e^{-w}(\star u)] = -\star [\partial e^{-w}\wedge\star u + e^{-w}\partial\star
u]\]
\[ = -\star[\partial e^{-w}\wedge\star u] + e^{-w}\dbar^* u =
e^{-w}\star[\partial w\wedge\star u] + e^{-w}\dbar^* u.\]
\noindent
With the corresponding expression
\[\dbar^* e^w u =  -e^w\star[\partial w\wedge\star u] + e^w\dbar^* u,\]
\noindent
we obtain
\begin{align} \langle\dbar^*e^{-w}u,\dbar^* e^w u\rangle =&  \langle\dbar^* u,
\dbar^* u\rangle + 2i\, \mathfrak{Im}\, \langle\star(\partial w\wedge\star
u),\dbar^* u\rangle  \notag \\ &- \langle (\partial w\wedge\star u),(\partial
w\wedge\star u) \rangle\notag, \end{align}
\noindent
where we have used the fact that the Hodge $\star$ is an isometry. \end{proof}
\begin{corollary}\label{cormin} Assuming $\langle\dbar w,\dbar w\rangle_{\Lambda^{0,1}}\le
1$, we have
\[-\mathfrak{Re}\,Q(e^{-w}u,e^w u)\le 2 \|u\|_{L^2(M,\Lambda^{p,q})}^2.\]
\end{corollary}
\begin{proof} The previous assertion implies
\[-\mathfrak{Re}\,Q(e^{-\epsilon w}u,e^{\epsilon w} u) = \epsilon^2 \left\{
\|\dbar w \wedge u\|^2 + \|\partial w \wedge\star u\|^2\right\} - Q(u,u)\]
\noindent
and since we have assumed $\langle\partial w,\partial
w\rangle_{\Lambda^{1,0}} = \langle\dbar w,\dbar
w\rangle_{\Lambda^{0,1}}\le 1$, (see Sect.\ \ref{intrin}) we have the
result by Cauchy-Schwarz and again the fact that the Hodge $\star$ is an
isometry.  \end{proof}
\begin{proof}[\bf{Proof of Theorem \ref{thm1.2}}] For arbitrary $f\in\dom
Q$, the computation in Prop.\ \ref{derive}
 gives
\begin{align}\|e^w P_t f\|_{L^2}^2 - \|e^w f\|_{L^2}^2 =& \int_0^t
\frac{d}{ds}\|e^w P_s f\|_{L^2}^2  ds \notag\\
=& -2\mathfrak{Re}\, \int_0^t ds\ Q (P_s f, e^{2w} P_s f).\label{int}\end{align}
Writing
\[ Q (P_s f, e^{2w} P_s f) = Q (e^{-w} e^w P_s f, e^w e^w P_s f)\]
\noindent
and applying Cor.\ \ref{cormin}, the integrand in \eqref{int} satisfies
\begin{equation}\label{corapp}-\mathfrak{Re}\, Q (P_s f, e^{2w} P_s f) \le \|e^w
P_s f\|_{L^2}^2,\end{equation}
\noindent
as usual, assuming that $\langle\dbar w,\dbar w\rangle_{\Lambda^{0,1}}\le 1$. It
follows that
\[ \|e^w P_t f\|_{L^2}^2 - \|e^w f\|_{L^2}^2 \le 2\int_0^t ds\ \|e^w P_s
f\|_{L^2}^2. \]
\noindent
Gronwall's inequality implies that
\[\|e^w P_t f\|_{L^2}^2 \le e^{2t} \|e^w f\|_{L^2}^2\]
and replacing $w$ by $\delta w$ we obtain $\|e^{\delta w} P_t f\|_{L^2} \le
e^{\delta^2 t} \|e^{\delta w} f\|_{L^2}$ by inspection in \eqref{corapp}. This
implies that
\[\|e^{\delta w} P_t e^{-\delta w}\|_{2\to 2} \le e^{\delta^2 t}\]
\noindent
since $f$ was arbitrary in the domain.

Now, for arbitrary $\alpha,\beta\in L^2$
\begin{align}|\langle \bid_BP_t\bid_A\alpha, \beta\rangle| &= \left|\langle
e^{\delta
w} P_t e^{-\delta w} e^{\delta w}\bid_A\alpha, e^{-\delta
w}\bid_B\beta\rangle\right|
\notag\\
&\le \|e^{\delta w} P_t e^{-\delta w} e^{\delta w}\bid_A\alpha\|_{L^2(M)}
\|e^{-\delta w}\bid_B\beta\|_{L^2(M)} \notag\\
&\le \|e^{\delta w} P_t e^{-\delta w}\|_{2\to 2} \|e^{\delta
w}\bid_A\alpha\|_{L^2(M)} \|e^{-\delta w}\bid_B\beta\|_{L^2(M)}. \notag\\
&\le e^{\delta^2 t} \|e^{\delta w}\bid_A\alpha\|_{L^2(M)} \|e^{-\delta
w}\bid_B\beta\|_{L^2(M)}. \notag\end{align}
For $\varepsilon>0$ choose a weight function $w$ like in the definition of
$\dint(A;B)$ above, with $\langle \dbar w,\dbar w\rangle_{\Lambda^{0,1}} \le
1$ and so that
$$
\dint(A;B)-\varepsilon \le \inf_B w-\sup_A w \quad \mbox{and} \quad \sup_A
w=0
$$
(we can achieve the latter by adding a suitable constant). This gives
$$
\inf_B w \ge \dint(A;B)-\varepsilon .
$$
Inserting gives
$$
|\langle \bid_BP_t\bid_A\alpha, \beta\rangle|\le
e^{\delta^2t}e^{-\delta(\dint(A;B)-\varepsilon)}\|\alpha\|\|\beta\|
$$
so that (since $\varepsilon$ is arbitrary)
$$
\| \bid_BP_t\bid_A\| \le
e^{\delta^2t}e^{-\delta \dint(A;B)} .
$$
\noindent
For $\dint(A;B)<\infty$, choose $\delta = \dint(A;B)/(2t)$.
\end{proof}
\begin{rem} In light of Prop.\ \ref{equiv}, we may replace $\dint$ with $d_{LB}$, making the necessary changes.\end{rem}

\subsection{Sobolev estimates for the Heat Operator}

Here we extend some $L^p$ results from the preceding treatment to Sobolev
spaces. First note that for $t>0$ and $k\in\mathbb N$ arbitrary, we have
$P_t:L^2 \to \dom\square^k$.
\begin{prop}\label{likespecs} For $t>0$ and $q>0$ we have
$$P_t:L^2(M,\Lambda^{p,q})\to
C^\infty(\bar M,\Lambda^{p,q}).$$
\end{prop}
\begin{proof} We will proceed by induction and use the Sobolev lemma,
(\ref{sobolem}) above. Fix $t>0$. For any $\alpha\in L^2$, since $\im
P_t\subset\dom\square$, and $(\square + 1)^{-1}: L^2 \to \dom\square$ is onto,
we may apply Thm.\ \ref{globalize} to the form $\alpha = (\square + 1)P_t\beta$,
$\beta\in\dom\square$, to obtain
\[ \| \zeta P_t\beta\|_{H^1} \lesssim \|\zeta' (\square + 1)P_t\beta\|_{L^2} +
\|(\square + 1)P_t\beta\|_{L^2}\lesssim \|\beta\|_{L^2},\]
\noindent
and conclude that $\im P_t\in H^1_{\rm loc}$. Furthermore, since $P_t$ is a
function of $\square$, they commute and we also have
\[ (\square+\bid)P_t\beta = P_t(\square+\bid)\beta \in H^1_{\rm loc} \qquad
(\alpha\in L^2).\]
\noindent
Assuming now that $(\square+\bid)P_t\beta\in H^{s-1}_{\rm loc}$, the same
theorem provides
\[\| \zeta P_t\beta\|_{H^s} \lesssim  \|\zeta' (\square + 1)P_t\beta\|_{H^{s-1}}
+ \|(\square + 1)P_t\beta\|_{L^2},\]
\noindent
so $P_t\beta\in H^s_{\rm loc}$. \end{proof}
We will need the following {\it a priori} estimate for $\square$, proven in our setting by a variation on the methods in \cite{KN, FK}, in Thm.\ 4.5 of \cite{P1}.
\begin{lemma}\label{soboreg}{\rm\bf (Kohn inequality)} If $M$ is a strongly pseudoconvex $G$-manifold on which $G$ acts freely by holomorphic transformations with compact quotient $\bar M/G$ and $q>0$, then for every integer $s\ge 0$ there exists
a positive constant $C_s$ so that
\[\| u\|_{H^{s+1}}\le C_s( \|\square u\|_{H^s}+ \|u\|_{L^2}), \quad
(u\in\dom\square \cap C^\infty(\bar M,\Lambda^{p,q}))\]
\noindent
uniformly.\end{lemma}
\begin{corollary}\label{bomb} For $t>0$ and $q>0$ we have $\im P_t\subset H^\infty(M, \Lambda^{p,q})$.
\end{corollary}
\begin{proof} Combining the results of Prop.\ \ref{likespecs} and Lemma
\ref{soboreg}, we have
\[\| u\|_{H^{s+1}}\le C_s( \|\square u\|_{H^s}+ \|u\|_{L^2}) \qquad (u\in\im
P_t)\]
\noindent
but $\im P_t \subset\dom\square^k$ for all powers of the Laplacian, so this
estimate can be iterated. Thus the estimates
\begin{equation}\label{red}\|\square^{k-s}u\|_{H^{s+1}}\lesssim
\|\square^{k-s+1} u\|_{H^s} + \|\square^{k-s}u\|_{L^2}, \quad
(s=1,2,\dots,k)\end{equation}
\noindent
hold for $u\in\im P_t$. These imply the result.\end{proof}
\begin{prop}\label{applied} If $M$ is as above, $t>0$, and $q>0$, then the heat operator $P_t$ is
bounded from $H^{-s}(\bar M,\Lambda^{p,q})\to H^{s}(M,\Lambda^{p,q})$ for any
positive integer $s$.\end{prop}
\begin{proof} First recall the following fact about Sobolev spaces on manifolds
with boundary from Remark 12.5 of \cite{LM}. For $s>0$, the dual space of
$H^s(M)$, denoted $H^{-s}(\bar M)$, consists of elements of
$H^{-s}(\widetilde{M})$ whose support is in $\bar M$. Now, from Cor.\ \ref{bomb}
we have that for all $s>0$, $P_t: L^2 \to H^s(M)$ continuously. Since $P_t$ is
self-adjoint, its domain can be extended to the dual of $H^s(M)$ so that
$P_t:H^{-s}(\bar M)\to L^2(M)$. The semigroup law $P_t^2=P_{2t}$ holds on
$C^\infty_c(M)\subset L^2(M)$, a dense subspace of all the $H^s(M)$,
$(s\in\mathbb R)$ so we may conclude that $P_{t}:H^{-s}(\bar M)\to H^s(M)$ for
all $s>0$.\end{proof}
\begin{rem} These results have three easy consequences.

\medskip

\noindent
1) For an operator norm estimate, we can put $u=P_t \alpha$ in the estimates \eqref{red} and telescope them to find that for $s\in\mathbb N$,
\[\| P_t \alpha \|_{H^s} \lesssim \sum_{k=0}^s \| \square^k P_t \alpha\|_{L^2}
\lesssim \sum_{k=0}^s t^{-k}  \|\alpha\|_{L^2},\]
\noindent
which yields an estimate analogous to that in Prop.\ \ref{ultra}.

\medskip

\noindent
2) Combining Cor.\ \ref{bomb} with Gagliardo-Nirenberg-Sobolev embeddings, {\it e.g.}
\[H^s(\mathbb R^n)\subset L^p(\mathbb R^n), \qquad p=\frac{2n}{n-2s}, \ \ 0\le s
< \frac{n}{2},\]
\noindent
\cite{BL}, one obtains results overlapping those of the previous sections in
$L^p$ spaces. With other such embeddings can obtain results for
$L^p$-Sobolev spaces.

\medskip

\noindent
3) One can continue the treatment in Sect.\ 6 of
\cite{P1} to obtain that, for $t>0$, the heat operator's Schwartz kernel $K_t\in
C^\infty(\bar M\times\bar M)$ and
\begin{equation}\label{GHilbSchm}\int_{\frac{M\times M}{G}} |K_t|^2 < \infty,
\qquad (t>0),\end{equation}
\noindent
noting that $\square$ and thus $K_t$ are $G$-invariant. When $G$ is unimodular,
\eqref{GHilbSchm} means that von Neumann's $G$-trace of $P_{2t}$ is finite.
\end{rem}

\begin{ack} It is a pleasure to thank Norbert Peyerimhoff for helpful
discussions. JJP would like to thank the Math Department at
TU-Chemnitz for two visits.\end{ack}

\end{document}